\documentclass[a4paper]{amsart}
\usepackage{amsfonts,amsmath,amssymb,amscd}
\usepackage[cp1251]{inputenc}
\usepackage[T2A]{fontenc}

\newtheorem{theorem}{Theorem}

\theoremstyle{remark}

\theoremstyle{definition}

\newtheorem{example}[theorem]{Example}

\def\C{\mathbb C}
\def\R{\mathbb R}
\def\T{\mathbb T}
\def\Z{\mathbb Z}

\newcommand{\mb}[1]{{\textbf {\textit#1}}}

\def\ge{\geqslant}
\begin{document}
\title{Hamiltonian minimal Lagrangian submanifolds\\in toric varieties}

\author{Andrey Mironov}
\address{Sobolev Institute of Mathematics, 4 Acad. Koptyug avenue, 630090 Novosibirsk, Russia,
\quad \emph{and}\newline\indent Laboratory of Geometric Methods in
Mathematical Physics, Moscow State University}
\email{mironov@math.nsc.ru}

\author{Taras Panov}
\address{Department of Mathematics and Mechanics, Moscow
State University
\newline\indent Institute for Theoretical and Experimental Physics,
Moscow, Russia,\quad \emph{and}
\newline\indent Institute for Information Transmission Problems,
Russian Academy of Sciences}
\email{tpanov@mech.math.msu.su}

\thanks{The first author was supported by grants МД-5134.2012.1 and НШ-544.2012.1 from the
President of Russia. The second author was supported by grants
МД-111.2013.1 and НШ-4995-2012.1 from the President of Russia and
RFBR grant~11-01-00694. Both authors were supported by RFBR
grant~12-01-92104-ЯФ, grants from Dmitri Zimin's `Dynasty'
foundation, and grant no.~2010-220-01-077 of the Government of
Russia.}

\maketitle

Hamiltonian minimality ($H$-minimality for short) for Lagrangian
submanifolds is a symplectic analogue of Riemannian minimality. A
Lagrangian immersion is called $H$-minimal if the variations of
its volume along all Hamiltonian vector fields are zero. This
notion was introduced in the work of Y.-G.~Oh~\cite{oh93} in
connection with the celebrated \emph{Arnold conjecture} on the
number of fixed points of a Hamiltonian symplectomorphism.

In~\cite{miro04} and~\cite{mi-pa} the authors defined and studied
a family of $H$-minimal Lagrangian submanifolds in~$\C^m$ arising
from intersections of real quadrics. Here we extend this
construction to define $H$-minimal submanifolds in toric
varieties.

The initial data of the construction is an intersection of $m-n$
Hermitian quadrics in~$\C^m$:
\begin{equation}\label{zgamma}
  \mathcal Z=\Bigl\{\mb z=(z_1,\ldots,z_m)\in\C^m\colon
  \sum_{k=1}^m\gamma_{jk}|z_k|^2=\delta_j\quad\text{for }
  j=1,\ldots,m-n\Bigr\}.
\end{equation}

We assume that the intersection is nonempty, nondegenerate and
rational; these conditions can be expressed in terms of the
coefficient vectors
$\gamma_k=(\gamma_{1k},\ldots,\gamma_{m-n,k})^t\in\R^{m-n}$ as
follows:
\begin{itemize}
\item[(a)] $\delta\in
\R_\ge\langle\gamma_1,\ldots,\gamma_m\rangle$ ($\delta$ is in the
cone generated by $\gamma_1,\ldots,\gamma_m$);

\item[(b)] if $\delta\in\R_\ge\langle
\gamma_{i_1},\ldots\gamma_{i_p}\rangle$, then $p\ge m-n$;

\item[(c)]
the vectors $\gamma_1,\ldots,\gamma_m$ span a lattice $L$ of full
rank in~$\R^{m-n}$.
\end{itemize}

Under these conditions, $\mathcal Z$ is a smooth
$(m+n)$-dimensional submanifold in $\C^m$, and
\[
  T_\varGamma=
  \bigl\{\bigr(e^{2\pi i\langle\gamma_1,\varphi\rangle},
  \ldots,e^{2\pi
  i\langle\gamma_m,\varphi\rangle}\bigl),\quad\varphi\in\R^{m-n}\bigr\}=
  \R^{m-n}/\displaystyle L^*
\]
is an $(m-n)$-dimensional torus. We represent elements of
$T_\varGamma$ by $\varphi\in\R^{m-n}$. We also define
\[
  D_\varGamma=({\textstyle\frac12}L^*)/L^*\cong(\Z_2)^{m-n}.
\]
Note that $D_\varGamma$ embeds canonically as a subgroup in
$T_\varGamma$.

Let $\mathcal R\subset\mathcal Z$ be the subset of real points,
which can be written by the same equations in real coordinates:
\[
  \mathcal R=\Bigl\{\mb u=(u_1,\ldots,u_m)\in\R^m\colon
  \sum_{k=1}^m\gamma_{jk}u_k^2=\delta_j\quad\text{for }
  j=1,\ldots,m-n\Bigr\}.
\]
We `spread' $\mathcal R$ by the action of $T_\varGamma$, that is,
consider the set of $T_\varGamma$-orbits through $\mathcal R$.
More precisely, we consider the map
\begin{align*}
  j\colon\mathcal R\times T_\varGamma &\longrightarrow \C^m,\\
  (\mb u,\varphi) &\mapsto \mb u\cdot\varphi=\bigl(u_1e^{2\pi
i\langle\gamma_1,\varphi\rangle},\ldots,u_me^{2\pi
i\langle\gamma_m,\varphi\rangle}\bigr)
\end{align*}
and observe that $j(\mathcal R\times T_\varGamma)\subset\mathcal
Z$. We let $D_\varGamma$ act on $\mathcal R_\varGamma\times
T_\varGamma$ diagonally; this action is free since it is free on
the second factor. The quotient
\[
  N=\mathcal R\times_{D_\varGamma} T_\varGamma
\]
is an $m$-dimensional manifold.

\begin{theorem}[\cite{miro04}]\label{hmin}
The map $j\colon\mathcal R\times T_\varGamma\to\C^m$ induces an
$H$-minimal Lagrangian immersion $i\colon N\looparrowright\C^m$.
\end{theorem}

Intersection of quadrics~\eqref{zgamma} is invariant with respect
to the diagonal action of the standard torus $\T^m\subset\C^m$.
The quotient $\mathcal Z/\T^m$ is identified with the set of
nonnegative solutions of the system of linear equations
$\sum_{k=1}^m\gamma_ky_k=\delta$. This set may be described as a
convex $n$-dimensional polyhedron
\begin{equation}\label{ptope}
  P=\bigl\{\mb x\in\R^n\colon\langle\mb a_i,\mb x\rangle+b_i\ge0\quad\text{for }
  i=1,\ldots,m\bigr\},
\end{equation}
where $(b_1,\ldots,b_m)$ is any solution and the vectors $\mb
a_1,\ldots,\mb a_m\in\R^n$ form the transpose of a basis of
solutions of the homogeneous system
$\sum_{k=1}^m\gamma_ky_k=\mathbf 0$.  We refer to $P$ as the
\emph{associated polyhedron} of the intersection of
quadrics~\eqref{zgamma}. The vector configurations
$\gamma_1,\ldots,\gamma_m$ and $\mb a_1,\ldots,\mb a_m$ are
\emph{Gale dual}.

Let $\Lambda$ denote the lattice of rank $n$ spanned by $\mb
a_1,\ldots,\mb a_m$. Polyhedron~\eqref{ptope} is called
\emph{Delzant} if, for any vertex $\mb x\in P$, the vectors $\mb
a_{i_1},\ldots,\mb a_{i_k}$ normal to the facets meeting at $\mb
x$ form a basis of the lattice~$\Lambda$. A Delzant $n$-polyhedron
is \emph{simple}, that is, there are exactly $n$ facets meeting at
each of its vertices.

\begin{theorem}[\cite{mi-pa}]
The immersion $i\colon N\looparrowright\C^m$ is an embedding if
and only if the associated polyhedron $P$ is Delzant.
\end{theorem}

Now we consider two sets of quadrics:
\begin{align*}
  \mathcal Z_{\varGamma}&=\Bigl\{\mb z\in\C^m\colon
  \sum\nolimits_{k=1}^m\gamma_k|z_k|^2=\mb c\Bigr\},\quad \gamma_k,\mb c\in\R^{m-n};\\
  \mathcal Z_{\varDelta}&=\Bigl\{\mb z\in\C^m\colon
  \sum\nolimits_{k=1}^m\delta_k|z_k|^2=\mb d\Bigr\},\quad \delta_k,\mb
  d\in\R^{m-\ell};
\end{align*}
such that $\mathcal Z_{\varGamma}$, $\mathcal Z_{\varDelta}$
\emph{and} $\mathcal Z_{\varGamma}\cap\mathcal Z_{\varDelta}$
satisfy conditions (a)--(c) above. Assume also that the polyhedra
associated with $\mathcal Z_{\varGamma}$, $\mathcal Z_{\varDelta}$
and $\mathcal Z_{\varGamma}\cap\mathcal Z_{\varDelta}$ are
Delzant.

The idea is to use the first set of quadrics to produce a toric
manifold $V$ via symplectic reduction, and then use the second set
of quadrics to define an $H$-minimal Lagrangian submanifold
in~$V$.

We define the real intersections of quadrics $\mathcal
R_{\varGamma}$, $\mathcal R_{\varDelta}$, the tori
$T_{\varGamma}\cong\T^{m-n}$, $T_{\varDelta}\cong\T^{m-\ell}$, and
the groups $D_{\varGamma}\cong\Z_2^{m-n}$,
$D_{\varDelta}\cong\Z_2^{m-\ell}$ as above.

We consider the toric variety $V$ obtained as the symplectic
quotient of $\C^m$ by the torus corresponding to the first set of
quadrics: $V=\mathcal Z_{\varGamma}/T_{\varGamma}$. It is a
K\"ahler manifold of real dimension~$2n$. The quotient $\mathcal
R_{\varGamma}/D_{\varGamma}$ is the set of real points of $V$ (the
fixed point set of the complex conjugation, or the real toric
manifold); it has dimension~$n$. Consider the subset of $\mathcal
R_{\varGamma}/D_{\varGamma}$ defined by the second set of
quadrics:
\[
  \mathcal S=(\mathcal R_{\varGamma}\cap\mathcal
  R_{\varDelta})/D_{\varGamma},
\]
we have $\dim\mathcal S=n+\ell-m$. Finally define the
$n$-dimensional submanifold of $V$:
\[
  N=\mathcal S\times_{D_{\varDelta}}T_{\varDelta}.
\]

\begin{theorem}
$N$ is an $H$-minimal Lagrangian submanifold in~$V$.
\end{theorem}
\begin{proof}
Let $\widehat V$ be the symplectic quotient of $V$ by the torus
corresponding to the second set of quadrics, that is, $\widehat
V=(V\cap\mathcal Z_\varDelta)/T_\varDelta=(\mathcal
Z_{\varGamma}\cap\mathcal Z_{\varDelta})/(T_{\varGamma}\times
T_{\varDelta})$. It is a toric manifold of real dimension
$2(n+\ell-m)$. The submanifold of real points
\[
  \widehat N=N/T_{\varDelta}=(\mathcal R_{\varGamma}\cap\mathcal
  R_{\varDelta})/(D_{\varGamma}\times D_{\varDelta})\hookrightarrow(\mathcal
  Z_{\varGamma}\cap\mathcal Z_{\varDelta})/(T_{\varGamma}\times
  T_{\varDelta})=\widehat V
\]
is the fixed point set of the complex conjugation, hence it is a
totally geodesic submanifold. In particular, $\widehat N$ is a
minimal submanifold in~$\widehat V$. According
to~\cite[Cor.~2.7]{dong07}, $N$ is an $H$-minimal submanifold
in~$V$.
\end{proof}

\begin{example}\

1. If $m-\ell=0$, i.e. $\mathcal Z_{\varDelta}=\varnothing$, then
$V=\C^m$ and we get the original construction of $H$-minimal
Lagrangian submanifolds $N$ in~$\C^m$.

2. If $m-n=0$, i.e. $\mathcal Z_{\varGamma}=\varnothing$, then $N$
is set of real points of~$V$. It is minimal (totally geodesic).

3. If $m-\ell=1$, i.e. $\mathcal Z_{\varDelta}\cong S^{2m-1}$,
then we get $H$-minimal Lagrangian submanifolds in $V=\C P^{m-1}$.
This subsumes many previously constructed families of projective
examples.
\end{example}

\enlargethispage{3\baselineskip}
\end{document}